\documentclass[]{amsart}

\usepackage[english]{babel}
\usepackage[utf8]{inputenc}
\usepackage{paralist}
\usepackage{amsmath, amsfonts, amssymb, amsthm}
\usepackage{color}
\usepackage{ulem}
\usepackage{tikz}

\newcommand{\CC}{\mathbb{C}}

\newcommand{\Hc}{\mathcal{H}}

\newcommand{\set}[1]{\left\{ #1 \right\}}
\newcommand{\setb}[1]{\left( #1 \right)}
\newcommand{\abs}[1]{\left| #1 \right|}

\newcommand{\gfr}{\mathfrak{g}}

\newcommand{\bino}[2]{\begin{pmatrix} #1 \\ #2 \end{pmatrix}}

\newtheorem{mymasterthm}{notForUse}
\theoremstyle{definition}

\theoremstyle{plain}
\newtheorem{mylemma}[mymasterthm]{Lemma}
\newtheorem{mythm}[mymasterthm]{Theorem}

\newtheorem{myprop}[mymasterthm]{Proposition}

\allowdisplaybreaks

\title{A Polynomial Variant of Diophantine Triples in Linear Recurrences}

\author[C. Fuchs]{Clemens Fuchs}
\author[S. Heintze]{Sebastian Heintze}
\thanks{Supported by Austrian Science Fund (FWF): I4406.}

\address{University of Salzburg\newline
	\indent Department of Mathematics\newline
	\indent Hellbrunnerstr. 34 \newline
	\indent A-5020 Salzburg, Austria}
\email{clemens.fuchs@sbg.ac.at, sebastian.heintze@sbg.ac.at}

\begin{document}
	
	\maketitle
	
	\begin{abstract}
		Let $ (G_n)_{n=0}^{\infty} $ be a polynomial power sum, i.e. a simple linear recurrence sequence of complex polynomials with power sum representation $ G_n = f_1\alpha_1^n + \cdots + f_k\alpha_k^n $ and polynomial characteristic roots $ \alpha_1,\ldots,\alpha_k $. For a fixed polynomial $ p $, we consider triples $ (a,b,c) $ of pairwise distinct non-zero polynomials such that $ ab+p, ac+p, bc+p $ are elements of $ (G_n)_{n=0}^{\infty} $. We will prove that under a suitable dominant root condition there are only finitely many such triples if neither $ f_1 $ nor $ f_1 \alpha_1 $ is a perfect square.
	\end{abstract}
	
	\section{Introduction}
	
	The study of Diophantine tuples started more than two thousand years ago and is well known in the meantime. Hereby, a Diophantine $ n $-tuple is a set $ \set{a_1,\ldots,a_n} $ of rational integers with the property that $ a_i a_j + 1 $ is a perfect square for all $ 1 \leq i < j \leq n $.
	Furthermore, one can consider so-called $ D(m) $-$ n $-tuples which are sets $ \set{a_1,\ldots,a_n} $ with the property that $ a_i a_j + m $ is a perfect square. For $ m = 1 $ we get the above defined Diophantine $ n $-tuples.
	One can consider many other variants of Diophantine tuples, e.g. algebraic integers or polynomials instead of integers, higher or perfect powers instead of squares, etc.;
	for a summary about Diophantine tuples and its variants we refer to \cite{dujella-web}.
	
	In \cite{dujella-fuchs-luca-2008} it is proved that the size of a set of complex polynomials with the property that the product of any two of them plus $ 1 $ is a perfect square is bounded above by $ 10 $.
	This bound is in \cite{dujella-jurasic-2010} reduced to $ 7 $.
	It is not clear what the expected true upper bound is.
	If we would add an arbitrary given polynomial $ p $ instead of $ 1 $, then in the general case there is no upper bound known so far. In the case of linear polynomials $ p $ there are some results in \cite{dujella-fuchs-tichy-2002} as well as in \cite{dujella-fuchs-walsh-2006} and for quadratic polynomials $ p $ in \cite{jurasic-2011}, but these results assume that the polynomials have integer coefficients.
	
	Since the sequence of squares can be written as a linear recurrence sequence, one can ask questions about existence of Diophantine tuples and finiteness of their number not only in the case of squares but also if we restrict $ a_i a_j + 1 $ to take values in an arbitrary fixed linear recurrence sequence. This situation has been considered in \cite{fuchs-luka-szalay-2008} for the first time.
	For instance, in \cite{fuchs-hutle-luca-2018} (see also the papers cited therein) the first author together with Hutle and Luca considered triples $ \setb{a,b,c} $ of positive integers satisfying $ 1 < a < b < c $ such that $ ab+1, ac+1, bc+1 $ are values in a linear recurrence sequence of Pisot type with Binet formula $ f_1\alpha_1^n + \cdots + f_k\alpha_k^n $. They proved three independent conditions under which there are only finitely many such triples. One of these conditions allows neither $ f_1 $ nor $ f_1 \alpha_1 $ to be a perfect square. This condition will be used in our statement, too.
	
	In the present paper we consider a function field variant of the Diophantine tuples taking values in recurrences. Namely, we study triples $ \setb{a,b,c} $ of pairwise distinct non-zero complex polynomials with the property that $ ab+1, ac+1, bc+1 $ are elements of a given linear recurrence sequence of polynomials. It will be proven that under some conditions on the recurrence sequence there are only finitely many such triples. In fact we are going to prove even more: the same result still holds if we add an arbitrary fixed polynomial $ p $ instead of $ 1 $.
	We mention that the same holds true if we replace $ \CC $ by an arbitrary algebraically closed field of characteristic $ 0 $.
	
	\section{Results}
	
	We call $ (G_n)_{n=0}^{\infty} $ a polynomial power sum if it is a simple linear recurrence sequence of complex polynomials with power sum representation
	\begin{equation}
		\label{p4-eq:polypowsum}
		G_n = f_1\alpha_1^n + \cdots + f_k\alpha_k^n
	\end{equation}
	such that $ \alpha_1, \ldots, \alpha_k \in \CC[X] $ are polynomials and $ f_1,\ldots,f_k \in \CC(X) $.
	Our main theorem that we are going to prove in this paper is the following statement:
	\begin{mythm}
		\label{p4-thm:mainresult}
		Let $ (G_n)_{n=0}^{\infty} $ be a polynomial power sum given as in \eqref{p4-eq:polypowsum}. Assume either that the order of this sequence is $ k \geq 3 $ and the dominant root condition $ \deg \alpha_1 > \deg \alpha_2 > \deg \alpha_3 \geq \deg \alpha_4 \geq \cdots \geq \deg \alpha_k $ is fulfilled, or that the order is $ k = 2 $ and we have $ \deg \alpha_1 > \deg \alpha_2 > 0 $.
		Moreover, let $ p \in \CC[X] $ be a given polynomial.
		If neither $ f_1 $ nor $ f_1 \alpha_1 $ is a square in $ \CC(X) $, then there are only finitely many triples $ \setb{a,b,c} $ of pairwise distinct non-zero polynomials such that $ ab+p, ac+p, bc+p $ are all elements of $ (G_n)_{n=0}^{\infty} $.
	\end{mythm}
	
	Note that we do not require that the linear recurrence sequence is non-degenerate, i.e. there may be two indices $ i $ and $ j $ such that $ \alpha_i $ and $ \alpha_j $ differ only be a constant factor. The only non-degeneracy properties we need are given by the dominant root condition.
	
	Moreover, the dominant root condition is really necessary in this situation. If we omit the dominant root condition, then the statement of the theorem does not hold any more as the following example illustrates:
	Let $ A_n, B_n, C_n $ be three linear recurrence sequences of polynomials. Then the products $ A_nB_n, A_nC_n, B_nC_n $ are also linear recurrence sequences of polynomials.
	Let $ A_nB_n = s_1 \sigma_1^n + \cdots + s_k \sigma_k^n $ be the Binet representation. Consider now a new linear recurrence sequence $ D_n $ generated from $ A_nB_n $ by the following procedure: Replace each summand $ s_i \sigma_i^n $ with the term
	\begin{equation*}
		\frac{1}{3} s_i \sigma_i^n + \frac{1}{3} s_i (\zeta_3 \sigma_i)^n +  \frac{1}{3} s_i (\zeta_3^2 \sigma_i)^n
	\end{equation*}
	where $ \zeta_3 $ is a primitive third root of unity. This new sequence $ D_n $ has for indices of the shape $ 3u $ the same values as $ A_nB_n $ and is zero otherwise. In other words $ D_{3u} = A_{3u}B_{3u} $ and $ D_{3u+1} = 0 = D_{3u+2} $.
	In the same manner we construct linear recurrence sequences $ E_n $ and $ F_n $ such that $ E_{3u+1} = A_{3u}C_{3u} $ and $ E_{3u} = 0 = E_{3u+2} $ as well as $ F_{3u+2} = B_{3u}C_{3u} $ and $ F_{3u} = 0 = F_{3u+1} $.
	Last but not least we define the sequence $ G_n := D_n + E_n + F_n + 1 $. Thus we have
	\begin{equation*}
		G_n =
		\begin{cases}
			A_{3u}B_{3u} + 1 &\text{ if } n=3u \\
			A_{3u}C_{3u} + 1 &\text{ if } n=3u+1 \\
			B_{3u}C_{3u} + 1 &\text{ if } n=3u+2
		\end{cases}.
	\end{equation*}
	First note that $ G_n $ has no dominant root.
	If we choose the simple sequences $ A_n,B_n,C_n $ in such a way that all characteristic roots are squares in $ \CC[X] $, then all characteristic roots of $ G_n $ are squares in $ \CC[X] $ as well.
	Thus $ f_1 $ is a square in $ \CC(X) $ if and only if $ f_1 \alpha_1 $ is a square in $ \CC(X) $.
	Moreover, choose $ A_n,B_n,C_n $ such that all occurring characteristic roots are pairwise distinct, non-constant and have pairwise no common root.
	Furthermore, let all coefficients in the Binet-formulas of $ A_n,B_n,C_n $ be non-constant polynomials without multiple roots. Assume that these coefficients are pairwise distinct and have pairwise no common root.
	Additionally, no complex number $ z \in \CC $ should be a root of both, an arbitrary root and an arbitrary coefficient.
	According to this construction no coefficient of a non-constant characteristic root of $ G_n $ is a square in $ \CC(X) $.
	Nevertheless, there are obviously infinitely many triples $ (a,b,c) $ such that $ ab+1, ac+1, bc+1 $ are all elements of $ (G_n)_{n=0}^{\infty} $.
	
	We remark that our preliminary assumptions are somewhat the opposite of those in \cite{fuchs-hutle-luca-2018} since there the characteristic polynomial is irreducible whereas in our case the characteristic polynomial splits in linear factors over the ground field.
	Concerning the conclusion, Theorem \ref{p4-thm:mainresult} can be seen as a function field analogue of Theorem 2 in \cite{fuchs-hutle-luca-2018}. We are unable to prove the result under the condition of Theorem 3 in \cite{fuchs-hutle-luca-2018}, i.e. for any large $ k $.
	
	Our result is ineffective in the sense that our method of proof does neither produce an upper bound for the number of solutions nor gives a method to actually locate them.
	
	\section{Preliminaries}
	
	For the convenience of the reader we give a short wrap-up of the notion of valuations and of the height that can e.g. also be found in \cite{fuchs-heintze-p2} and \cite{fuchs-heintze-p3}:
	
	For $ c \in \CC $ and $ f(X) \in \CC(X) $ where $ \CC(X) $ is the rational function field over $ \CC $ denote by $ \nu_c(f) $ the unique integer such that $ f(X) = (X-c)^{\nu_c(f)} p(X) / q(X) $ with $ p(X),q(X) \in \CC[X] $ such that $ p(c)q(c) \neq 0 $. Further denote by $ \nu_{\infty}(f) = \deg q - \deg p $ if $ f(X) = p(X) / q(X) $.
	These functions $ \nu $ are up to equivalence all valuations in $ \CC(X) $.
	If $ \nu_c(f) > 0 $, then $ c $ is called a zero of $ f $, and if $ \nu_c(f) < 0 $, then $ c $ is called a pole of $ f $.
	For a finite extension $ F $ of $ \CC(X) $ each valuation in $ \CC(X) $ can be extended to no more than $ [F : \CC(X)] $ valuations in $ F $. This again gives all valuations in $ F $.
	Both, in $ \CC(X) $ as well as in $ F $ the sum-formula
	\begin{equation*}
		\sum_{\nu} \nu(f) = 0
	\end{equation*}
	holds, where $ \sum_{\nu} $ means that the sum is taken over all valuations in the considered function field.
	Each valuation in a function field corresponds to a place and vice versa.
	The places can be thought as the equivalence classes of valuations.
	Moreover, we write $ \deg f = -\nu_{\infty}(f) $ for all $ f(X) \in \CC(X) $.
	
	Furthermore, the proof in the next section will take use of height functions in function fields. Let us therefore define the height of an element $ f \in F^* $ by
	\begin{equation*}
		\Hc(f) := - \sum_{\nu} \min \setb{0, \nu(f)} = \sum_{\nu} \max \setb{0, \nu(f)}
	\end{equation*}
	where the sum is taken over all valuations on the function field $ F / \CC $. Additionally we define $ \Hc(0) = \infty $.
	This height function satisfies some basic properties that are listed in the lemma below, which is proven in \cite{fuchs-karolus-kreso-2019}:
	
	\begin{mylemma}
		\label{p4-lemma:heightproperties}
		Denote as above by $ \Hc $ the height on $ F/\CC $. Then for $ f,g \in F^* $ the following properties hold:
		\begin{enumerate}[a)]
			\item $ \Hc(f) \geq 0 $ and $ \Hc(f) = \Hc(1/f) $,
			\item $ \Hc(f) - \Hc(g) \leq \Hc(f+g) \leq \Hc(f) + \Hc(g) $,
			\item $ \Hc(f) - \Hc(g) \leq \Hc(fg) \leq \Hc(f) + \Hc(g) $,
			\item $ \Hc(f^n) = \abs{n} \cdot \Hc(f) $,
			\item $ \Hc(f) = 0 \iff f \in \CC^* $,
			\item $ \Hc(A(f)) = \deg A \cdot \Hc(f) $ for any $ A \in \CC[T] \setminus \set{0} $.
		\end{enumerate}
	\end{mylemma}
	
	When proving our theorem, we will use the following function field analogue of the Schmidt subspace theorem. A proof for this proposition can be found in \cite{zannier-2008}:
	\begin{myprop}[Zannier]
		\label{p4-prop:functionfieldsubspace}
		Let $ F/\CC $ be a function field in one variable, of genus $ \gfr $, let $ \varphi_1, \ldots, \varphi_n \in F $ be linearly independent over $ \CC $ and let $ r \in \set{0,1, \ldots, n} $. Let $ S $ be a finite set of places of $ F $ containing all the poles of $ \varphi_1, \ldots, \varphi_n $ and all the zeros of $ \varphi_1, \ldots, \varphi_r $. Put $ \sigma = \sum_{i=1}^{n} \varphi_i $. Then
		\begin{equation*}
			\sum_{\nu \in S} \left( \nu(\sigma) - \min_{i=1, \ldots, n} \nu(\varphi_i) \right) \leq \bino{n}{2} (\abs{S} + 2\gfr - 2) + \sum_{i=r+1}^{n} \Hc (\varphi_i).
		\end{equation*}
	\end{myprop}
	
	\section{Proof}
	
	In this section we will prove Theorem \ref{p4-thm:mainresult}. Before we begin the proof, let us remark that we will use some ideas of \cite{fuchs-hutle-luca-2018} that are quite useful also in our situation.
	
	\begin{proof}[Proof of Theorem \ref{p4-thm:mainresult}]
		We are going to prove the statement indirectly and assume therefore that there are infinitely many triples with the required properties.
		First note that without loss of generality we can assume that for a still infinite set of triples the inequality $ \deg a \leq \deg b \leq \deg c $ holds.
		For any such triple exist non-negative integers $ x,y,z $ such that
		\begin{equation}
			\label{p4-eq:defrelations}
			ab+p = G_x, \qquad ac+p = G_y, \qquad bc+p = G_z.
		\end{equation}
		If all three parameters $ x,y,z $ were bounded by a constant, then there can be only finitely many triples $ (a,b,c) $. Hence it must hold that $ \max \set{x,y,z} \rightarrow \infty $.
		
		Since $ \alpha_1 $ is the dominant root, for large enough $ n $, i.e. for $ n \geq n_0 $, the degree satisfies the (in)equality
		\begin{equation}
			\label{p4-eq:degformula}
			\deg G_n = \deg f_1 + n \deg \alpha_1 > \deg p.
		\end{equation}
		Now, as $ \max \set{x,y,z} \rightarrow \infty $, we get $ \max \set{\deg G_x, \deg G_y, \deg G_z} \rightarrow \infty $. Thus $ \deg c \rightarrow \infty $, and consequently we have $ z \rightarrow \infty $ as well as $ y \rightarrow \infty $.
		Therefore for a still infinite subset of triples we can assume that both, $ z $ and $ y $, are not smaller than $ n_0 $, which implies that equation \eqref{p4-eq:degformula} is applicable.
		
		Recalling the ordering $ \deg a \leq \deg b \leq \deg c $, we get $ \deg G_x \leq \deg G_y \leq \deg G_z $. Now again using \eqref{p4-eq:degformula} and taking into account that $ a,b,c $ are pairwise distinct yields $ x < y < z $.
		
		Let us for the moment assume that $ x $ is bounded by a constant. So $ G_x $ is the same fixed value for infinitely many triples. This implies that $ a $ and $ b $ are fixed for infinitely many triples. Therefore
		\begin{equation*}
			bG_y - aG_z = abc + bp - abc - ap = (b-a)p
		\end{equation*}
		is constant for infinitely many triples. Dividing by $ f_1 $ and using the power sum representation of the recurrence sequence, we can rewrite this equation in the form
		\begin{equation}
			\label{p4-eq:xunbnd1steq}
			(b - a\alpha_1^{z-y}) \alpha_1^y = - b\frac{f_2}{f_1} \alpha_2^y - \cdots - b\frac{f_k}{f_1} \alpha_k^y + a\frac{f_2}{f_1} \alpha_2^z + \cdots + a\frac{f_k}{f_1} \alpha_k^z + \frac{(b-a)p}{f_1}.
		\end{equation}
		If the left hand side of equation \eqref{p4-eq:xunbnd1steq} is non-zero, then it has degree at least $ y \deg \alpha_1 $.
		However, the degree of the right hand side is at most $ C_0 + z \deg \alpha_2 $, where $ C_0 $ is a constant.
		Since $ \deg G_y = \deg a + \deg c $ and $ \deg G_z = \deg b + \deg c $ the equation
		\begin{equation*}
			(z-y) \deg \alpha_1 = \deg G_z - \deg G_y = \deg b - \deg a
		\end{equation*}
		is satisfied which implies that $ \rho := z-y $ is constant.
		Consequently, the degree of the right hand side of equation \eqref{p4-eq:xunbnd1steq} is at most $ C_1 + y \deg \alpha_2 $ for a new constant $ C_1 $.
		The only way this can work is that both sides of equation \eqref{p4-eq:xunbnd1steq} are zero. Therefore, $ b = a \alpha_1^{\rho} $.
		Considering the equation
		\begin{equation*}
			(b-a)p = bG_y - aG_z = b \left( f_1\alpha_1^y + \cdots + f_k\alpha_k^y \right) - a \left( f_1\alpha_1^{y+\rho} + \cdots + f_k\alpha_k^{y+\rho} \right),
		\end{equation*}
		dividing this by $ a $ and replacing $ b $ by $ a \alpha_1^{\rho} $ yields
		\begin{equation}
			\label{p4-eq:xunbnd2ndeq}
			(\alpha_1^{\rho} - 1) p = \alpha_1^{\rho} \left( f_2\alpha_2^y + \cdots + f_k\alpha_k^y \right) - \left( f_2\alpha_2^{y+\rho} + \cdots + f_k\alpha_k^{y+\rho} \right).
		\end{equation}
		The left hand side of equation \eqref{p4-eq:xunbnd2ndeq} has constant degree whereas the degree of the right hand side is $ \rho \deg \alpha_1 + \deg f_2 + y \deg \alpha_2 $. This is a contradiction for large $ y $, implying that $ x $ cannot be bounded by a constant.
		
		Overall, there is a still infinite set of triples such that all three indices $ x,y,z $ are always greater than an arbitrary fixed constant. In particular, we can assume that no index is smaller than $ n_0 $.
		
		We have already mentioned above, that $ x < y < z $. In the next step it will be proven that $ z $ can not grow much faster than $ x $. For doing so let
		\begin{equation*}
			g := \gcd \setb{G_y-p, G_z-p}
		\end{equation*}
		be the greatest common divisor of these two polynomials.
		Now we distinguish between two cases.
		Firstly, assume $ y \leq \kappa z $, where $ \kappa $ is a rational number in the interval $ (0,1) $ which will be determined later.
		It holds that
		\begin{align*}
			\deg g &\leq \deg (G_y-p) = \deg f_1 + y \deg \alpha_1 \\
			&\leq \deg f_1 + \kappa z \deg \alpha_1 \leq C_2 + \kappa z \deg \alpha_1.
		\end{align*}
		Secondly, we assume in the other case that $ y > \kappa z $.
		Thus we have $ z-y < z - \kappa z = (1-\kappa) z $.
		By the definition of $ g $ as greatest common divisor, it immediately follows that $ g $ is also a divisor of $ (G_z-p) - \alpha_1^{z-y} (G_y-p) $ and therefore
		\begin{align*}
			\deg g &\leq \deg \left( (G_z-p) - \alpha_1^{z-y} (G_y-p) \right) \\
			&= \deg \left( (f_2 \alpha_2^z + \cdots + f_k \alpha_k^z - p) - (f_2 \alpha_1^{z-y} \alpha_2^y + \cdots + f_k \alpha_1^{z-y} \alpha_k^y - \alpha_1^{z-y} p) \right) \\
			&= \deg f_2 + (z-y) \deg \alpha_1 + y \deg \alpha_2 \\
			&\leq \deg f_2 + (z-y) \deg \alpha_1 + y (\deg \alpha_1 - 1) \\
			&= C_3 + z \deg \alpha_1 - y
			< C_3 + z \deg \alpha_1 - \kappa z
			= C_3 + z (\deg \alpha_1 - \kappa).
		\end{align*}
		We want to choose $ \kappa $ in such a way that $ \kappa \deg \alpha_1 = \deg \alpha_1 - \kappa $. Hence we set
		\begin{equation*}
			\kappa = \frac{\deg \alpha_1}{1 + \deg \alpha_1}
		\end{equation*}
		which yields in both of our cases
		\begin{equation}
			\label{p4-eq:gcdbound}
			\deg g \leq C_4 + z \kappa \deg \alpha_1.
		\end{equation}
		If we denote $ \widetilde{g} = \gcd \setb{G_x-p, G_z-p} $, then $ G_z-p $ is a divisor of $ g \widetilde{g} $ since $ c \mid g $, $ b \mid \widetilde{g} $ and $ G_z-p = bc $. This gives us
		\begin{align*}
			\deg f_1 + x \deg \alpha_1 &= \deg (f_1 \alpha_1^x) = \deg G_x = \deg (G_x - p) \geq \deg \widetilde{g} \\
			&\geq \deg (G_z-p) - \deg g = \deg f_1 + z \deg \alpha_1 - \deg g
		\end{align*}
		as well as by using inequality \eqref{p4-eq:gcdbound} that
		\begin{equation*}
			x \deg \alpha_1 \geq z \deg \alpha_1 - \deg g \geq z \deg \alpha_1 - C_4 - z \kappa \deg \alpha_1
		\end{equation*}
		and
		\begin{equation}
			\label{p4-eq:samegrowth}
			x \geq z - z \kappa - \frac{C_4}{\deg \alpha_1} = (1-\kappa) z - C_5 > C_6 z.
		\end{equation}
		Thus $ z $ is bounded above by $ x/C_6 $. This fact means that the three indices grow with a similar rate.
		
		Combining the three equations in \eqref{p4-eq:defrelations}, we can express the polynomials $ a,b,c $ by elements of the linear recurrence sequence $ (G_n)_{n=0}^{\infty} $ in the following way:
		\begin{align*}
			a &= \frac{\sqrt{G_x-p} \cdot \sqrt{G_y-p}}{\sqrt{G_z-p}}, \\
			b &= \frac{\sqrt{G_x-p} \cdot \sqrt{G_z-p}}{\sqrt{G_y-p}}, \\
			c &= \frac{\sqrt{G_y-p} \cdot \sqrt{G_z-p}}{\sqrt{G_x-p}}.
		\end{align*}
		By using the square root symbol in an equation we mean that the equation holds for a suitable choose of one of the two possible square roots, which differ only by the factor $ -1 $. This choose can vary from one equation to another.
		
		For this reason, we aim for rewriting the expression $ \sqrt{G_n-p} $ in a more suitable manner.
		This will be done by applying the (formal) multinomial series expansion to the power sum representation of our recurrence sequence:
		\begin{align*}
			\sqrt{G_n-p} &= \sqrt{f_1\alpha_1^n + \cdots + f_k\alpha_k^n - p} \\
			&= \sqrt{f_1} \alpha_1^{n/2} \sqrt{1 + \frac{f_2}{f_1} \left( \frac{\alpha_2}{\alpha_1} \right)^n + \cdots + \frac{f_k}{f_1} \left( \frac{\alpha_k}{\alpha_1} \right)^n - \frac{p}{f_1} \left( \frac{1}{\alpha_1} \right)^n} \\
			&= \sqrt{f_1} \alpha_1^{n/2} \sum_{h_1,\ldots,h_k=0}^{\infty} \gamma_{h_1,\ldots,h_k} \left( \frac{-p}{f_1} \right)^{h_1} \left( \frac{1}{\alpha_1} \right)^{nh_1} \left( \prod_{i=2}^{k} \left( \frac{f_i}{f_1} \right)^{h_i} \left( \frac{\alpha_i}{\alpha_1} \right)^{nh_i} \right) \\
			&= \sqrt{f_1} \alpha_1^{n/2} \sum_{h_1,\ldots,h_k=0}^{\infty} t_{h_1,\ldots,h_k}^{(n)} = \sqrt{f_1} \alpha_1^{n/2} \sum_{\underline{h}=0}^{\infty} t_{\underline{h}}^{(n)}
		\end{align*}
		where we use the notation $ \underline{h} = (h_1,\ldots,h_k) $ and
		\begin{equation*}
			t_{h_1,\ldots,h_k}^{(n)} = \gamma_{h_1,\ldots,h_k} \left( \frac{-p}{f_1} \right)^{h_1} \left( \frac{1}{\alpha_1} \right)^{nh_1} \left( \prod_{i=2}^{k} \left( \frac{f_i}{f_1} \right)^{h_i} \left( \frac{\alpha_i}{\alpha_1} \right)^{nh_i} \right).
		\end{equation*}
		
		The next step is now to calculate a lower bound for the valuation $ \nu_{\infty} $ of the quantity above, which we will need later on.
		Since $ \gamma_{h_1,\ldots,h_k} \in \CC $, we get
		\begin{align*}
			\nu_{\infty} \left( t_{\underline{h}}^{(n)} \right) &= h_1 \left( n \deg \alpha_1 + \deg \frac{f_1}{p} \right) + \sum_{i=2}^{k} h_i \left( n (\deg \alpha_1 - \deg \alpha_i) + \deg \frac{f_1}{f_i} \right) \\
			&\geq h_1 (n + C_7) +  \sum_{i=2}^{k} h_i (n + C_7)
			= \left( \sum_{i=1}^{k} h_i \right) (n + C_7)
		\end{align*}
		where $ C_7 = \min \set{\deg f_1 - \deg f_2, \ldots, \deg f_1 - \deg f_k, \deg f_1 - \deg p} $.
		Note that for our purpose we can assume $ n + C_7 > 0 $.
		
		Combining the representations of $ a,b,c $ and $ \sqrt{G_n-p} $ we have so far, yields the following representation of the product $ abc $ of the elements in a triple:
		\begin{align*}
			abc &= \sqrt{G_x-p} \sqrt{G_y-p} \sqrt{G_z-p} \\
			&= f_1^{3/2} \alpha_1^{(x+y+z)/2} \sum_{\underline{h}^{(x)}=0}^{\infty} t_{\underline{h}^{(x)}}^{(x)} \sum_{\underline{h}^{(y)}=0}^{\infty} t_{\underline{h}^{(y)}}^{(y)} \sum_{\underline{h}^{(z)}=0}^{\infty} t_{\underline{h}^{(z)}}^{(z)} \\
			&= f_1^{3/2} \alpha_1^{(x+y+z)/2} \sum_{\underline{h}^{(x)}, \underline{h}^{(y)}, \underline{h}^{(z)}=0}^{\infty} t_{\underline{h}^{(x)}}^{(x)} t_{\underline{h}^{(y)}}^{(y)} t_{\underline{h}^{(z)}}^{(z)}.
		\end{align*}
		For the valuation we get the lower bound
		\begin{align*}
			\nu_{\infty} \left( t_{\underline{h}^{(x)}}^{(x)} t_{\underline{h}^{(y)}}^{(y)} t_{\underline{h}^{(z)}}^{(z)} \right) &= \nu_{\infty} \left( t_{\underline{h}^{(x)}}^{(x)} \right) + \nu_{\infty} \left( t_{\underline{h}^{(y)}}^{(y)} \right) + \nu_{\infty} \left( t_{\underline{h}^{(z)}}^{(z)} \right) \\
			&\geq \left( \sum_{i=1}^{k} h_i^{(x)} + \sum_{i=1}^{k} h_i^{(y)} + \sum_{i=1}^{k} h_i^{(z)} \right) (x + C_7).
		\end{align*}
		
		Let $ J > 0 $ be a number to be fixed later.
		Then there exists a natural number $ L $, depending on $ J $, such that
		\begin{multline*}
			abc = f_1^{3/2} \alpha_1^{(x+y+z)/2} t_1^{(x,y,z)} + \cdots + f_1^{3/2} \alpha_1^{(x+y+z)/2} t_L^{(x,y,z)} \\
			+ \sum_{\nu_{\infty} \left( t_{\underline{h}^{(x)}}^{(x)} t_{\underline{h}^{(y)}}^{(y)} t_{\underline{h}^{(z)}}^{(z)} \right) \geq J(x + C_7)} f_1^{3/2} \alpha_1^{(x+y+z)/2} t_{\underline{h}^{(x)}}^{(x)} t_{\underline{h}^{(y)}}^{(y)} t_{\underline{h}^{(z)}}^{(z)}.
		\end{multline*}
		Define $ \varphi_0 = abc $ as well as
		\begin{equation*}
			\varphi_j = - f_1^{3/2} \alpha_1^{(x+y+z)/2} t_j^{(x,y,z)}
		\end{equation*}
		for $ j=1,\ldots,L $. Furthermore, put
		\begin{equation*}
			\sigma = \sum_{j=0}^{L} \varphi_j.
		\end{equation*}
		Let $ S $ be a finite set of places of $ F = \CC(X, \sqrt{\alpha_1}, \sqrt{f_1}) $ containing all places lying over the zeros of $ \alpha_1,\ldots,\alpha_k $, the zeros and poles of $ f_1,\ldots,f_k $, the zeros of $ p $, and all infinite places.
		Note that $ F $ contains always both possible values of the square roots if we require that one of them is contained since they differ only by the factor $ -1 $.
		By applying Proposition \ref{p4-prop:functionfieldsubspace}, if $ \varphi_0, \ldots, \varphi_L $ are linearly independent over $ \CC $, we get the inequality
		\begin{align*}
			\nu_{\infty} (\sigma) - \min_{i=0, \ldots, L} \nu_{\infty} (\varphi_i) &\leq C_8 + \Hc(abc) \\
			&\leq C_8 + C_9 \deg (abc) \\
			&= C_8 + \frac{C_9}{2} \deg (a^2b^2c^2) \\
			&= C_8 + \frac{C_9}{2} \deg ((G_x-p)(G_y-p)(G_z-p)) \\
			&= C_8 + \frac{C_9}{2} (3 \deg f_1 + (x+y+z) \deg \alpha_1) \\
			&\leq C_{10} + C_{11} z \deg \alpha_1.
		\end{align*}
		In order to get also a lower bound for this expression we look at
		\begin{equation*}
			\min_{i=0, \ldots, L} \nu_{\infty} (\varphi_i) \leq \nu_{\infty} (abc) = - \deg (abc) = - \frac{3}{2} \deg f_1 - \frac{x+y+z}{2} \deg \alpha_1,
		\end{equation*}
		which gives us
		\begin{align*}
			\nu_{\infty} (\sigma) - \min_{i=0, \ldots, L} \nu_{\infty} (\varphi_i) &\geq J (x+C_7) + \nu_{\infty} \left( f_1^{3/2} \alpha_1^{(x+y+z)/2} \right) - \min_{i=0, \ldots, L} \nu_{\infty} (\varphi_i) \\
			&\geq J (x+C_7).
		\end{align*}
		Now, recalling inequality \eqref{p4-eq:samegrowth}, we compare the lower with the upper bound to get
		\begin{equation*}
			J (x+C_7) \leq C_{10} + C_{11} z \deg \alpha_1 \leq C_{10} + \frac{C_{11}}{C_6} x \deg \alpha_1.
		\end{equation*}
		Therefore we set $ J := 1 + \frac{C_{11}}{C_6} \deg \alpha_1 $ and note that the right hand side does not depend on $ J $.
		Plugging this definition into the last inequality yields
		\begin{equation*}
			x \leq C_{10} - JC_7
		\end{equation*}
		which is a contradiction since we have already proven that $ x $ cannot be bounded by a constant and that therefore we may assume that all three indices are greater than any fixed constant.
		
		Thus $ \varphi_0, \ldots, \varphi_L $ must be linearly dependent over $ \CC $. Without loss of generality we may assume that $ \varphi_1, \ldots, \varphi_L $ are linearly independent since otherwise we could group them together before doing the previous step.
		Hence, in a relation of linear dependence, there must be a non-zero coefficient in front of $ abc $. So we can write
		\begin{equation*}
			abc = \sum_{j=1}^{L} \lambda_j f_1^{3/2} \alpha_1^{(x+y+z)/2} t_j^{(x,y,z)} = f_1^{3/2} \alpha_1^{(x+y+z)/2} \sum_{j=1}^{L} \lambda_j t_j^{(x,y,z)}
		\end{equation*}
		for $ \lambda_j \in \CC $.
		
		We distinguish between two cases considering the parity of $ x+y+z $.
		If $ x+y+z $ is even, then we have
		\begin{equation*}
			f_1^{1/2} = \frac{abc}{f_1 \alpha_1^{(x+y+z)/2} \sum_{j=1}^{L} \lambda_j t_j^{(x,y,z)}} \in \CC(X)
		\end{equation*}
		which contradicts the assumption in our theorem that $ f_1 $ is no square in $ \CC(X) $.
		When $ x+y+z $ is odd, we get
		\begin{equation*}
			(f_1 \alpha_1)^{1/2} = \frac{abc}{f_1 \alpha_1^{(x+y+z-1)/2} \sum_{j=1}^{L} \lambda_j t_j^{(x,y,z)}} \in \CC(X)
		\end{equation*}
		which contradicts the assumption in our theorem that $ f_1 \alpha_1 $ is no square in $ \CC(X) $.
		All in all there can be only finitely many triples satisfying the required properties.
	\end{proof}

\end{document}